\documentclass[12pt,a4paper]{article}
\usepackage{indentfirst}
\usepackage{amsmath,amssymb,amsfonts,amsthm,graphics}
\usepackage{makeidx}
\newtheorem{thm}{Theorem}[section]

\newtheorem{lem}{Lemma}[section]

\newtheorem{pro}{Proposition}[section]

\theoremstyle{definition}

\def\-{\mbox{--}}

\begin{document}
\title{\Large\bf Rainbow connection number\\[1mm] and the number
of blocks\footnote{Supported by NSFC No.11071130.}}
\author{\small  Xueliang~Li, Sujuan~Liu\\
\small Center for Combinatorics and LPMC-TJKLC\\
\small Nankai University, Tianjin 300071, China\\
\small lxl@nankai.edu.cn; sjliu0529@126.com}
\date{}
\maketitle
\begin{abstract}
An edge-colored graph $G$ is {\em rainbow connected} if every pair
of vertices of $G$ are connected by a path whose edges have distinct
colors. The {\em rainbow connection number} $rc(G)$ of $G$ is
defined to be the minimum integer $t$ such that there exists an
edge-coloring of $G$ with $t$ colors that makes $G$ rainbow
connected. For a graph $G$ without any cut vertex, i.e., a
2-connected graph, of order $n$, it was proved that $rc(G)\leq
\lceil\frac n 2 \rceil$ and the bound is tight. In this paper, we
prove that for a connected graph $G$ of order $n$ with cut vertices,
$rc(G) \leq\frac{n+r-1} 2$, where $r$ is the number of blocks of $G$
with even orders, and the upper bound is tight. Moreover, we also
obtain a tight upper bound for a bridgeless graph, i.e., a
2-edge-connected graph.



\end{abstract}

\section{Introduction}

All graphs considered in this paper are simple, finite and
undirected. For notation and terminology not defined here, we refer
to \cite{bondy2008graph}. In an edge-colored graph $G$, a path is
called a {\em rainbow path} if the colors of its edges are distinct.
The graph $G$ is called {\em rainbow connected} if every pair of
vertices are connected by at least one rainbow path in $G$. An
edge-coloring of a connected graph $G$ that makes $G$ rainbow
connected is called {\em a rainbow edge-coloring} of $G$. The
minimum number of colors required to rainbow color $G$ is called the
{\em rainbow connection number} of $G$, denoted by $rc(G)$. If a
graph $G$ has an edge-coloring $c$ and $G^\prime$ is a subgraph of
$G$, $c(G^\prime)$ denotes the set of colors appeared in $G^\prime$.
An edge-coloring using $k$ colors is addressed as a
$k$-edge-coloring. If $P$ is a path, the length of $P$ is denoted by
$\ell(P)$.

Let $G^\prime$ be a subgraph of a graph $G$. An {\em ear} of
$G^\prime$ in $G$ is a nontrivial path in $G$ whose end vertices lie
in $G^\prime$ but whose internal vertices are not. An {\em ear
decomposition} of a 2-connected graph $G$ is a sequence of subgraphs
$G_0, G_1, \cdots, G_k$ of $G$ satisfying that (1) $G_0$ is a cycle
of $G$; (2) $G_i=G_{i-1}\bigcup P_{i} \ (1\leq i\leq k)$, where
$P_{i}$ is an ear of $G_{i-1}$ in $G$; (3) $G_{i-1}(1\leq i\leq k)$
is a proper subgraph of $G_i$; (4) $G_k=G$. If $\ell(P_1)\geq
\ell(P_2)\geq\cdots \geq \ell(P_k)$, we say that the ear
decomposition is nonincreasing. From the above definition, every
graph $G_i$ in an ear decomposition is 2-connected.

A $block$ of a graph $G$ is a maximal connected subgraph of $G$ that
does not have any cut vertex. So every block of a nontrivial
connected graph is either a $K_2$ or a 2-connected subgraph. All the
blocks of a graph $G$ form a {\em block decomposition} of $G$. A
block $B$ is called an {\it even (odd) block} if the order of $B$ is
even (odd).

Let $c$ be a rainbow $k$-edge-coloring of a connected graph $G$. If a
rainbow path $P$ in $G$ has length $k$, we call $P$ $a$ $complete$
$rainbow$ $path$; otherwise, it is $an \ incomplete$ $rainbow$
$path$. A rainbow edge-coloring $c$ of $G$ is $incomplete$ if for any
vertex $u\in V(G)$, $G$ has at most one vertex $v$ such that all the
rainbow paths between $u$ and $v$ are complete; otherwise, it is
$complete$.

The definition of a rainbow coloring was introduced by Chartrand et
al. in \cite{Chartrand2008graph}. For more knowledge, we refer to
\cite{LSS2012connected, Li2012connected}. In
\cite{chakraborty2009hardness}, it was proved that computing the
rainbow connection number of a graph is $NP$-hard. Hence, tight
upper bounds of the rainbow connection number for a connected graph
have been a subject of investigation. The authors of
\cite{Chandran2011dominating} proved that
$rc(G)\leq3n/(\delta+1)+3$, where $\delta$ is the minimum degree of
the connected graph $G$, and the authors of \cite{bcrr, dongli}
obtained some upper bound of the rainbow connection number in term
of radius and bridges. For 2-connected graphs, there exist the
following results.

\begin{lem} \cite{Li2011connected} \label{lem1} Let $G$ be a
Hamiltonian graph of order $n \ (n\geq 3)$. Then $G$ has an
incomplete $\lceil\frac{n}{2}\rceil$-rainbow coloring, i.e.,
$rc(G)\leq\lceil\frac{n}{2}\rceil$.
\end{lem}

\begin{lem}\cite{Li2011connected} \label{lem2} Let $G$ be a
$2$-connected non-Hamiltonian graph of order $n \ (n\geq 4)$.
If $G$ has at most one ear with length 2 in a nonincreasing
ear decomposition, then $G$ has a incomplete $\lceil\frac{n}
{2}\rceil$-rainbow coloring, i.e., $rc(G)\leq\lceil\frac{n}{2}\rceil$.
\end{lem}

\begin{thm}\cite{Li2011connected, ekstein} \label{thm:2con}
Let $G$ be a $2$-connected graph of order $n \ (n\geq 3)$.
Then $rc(G)\leq\lceil\frac{n}{2}\rceil$, and the upper bound is
tight for $n\geq 4$.
\end{thm}

\begin{pro}\cite{caro2008rainbow} \label{pro1}
If $G$ is a connected bridgeless (2-edge-connected) graph with $n$
vertices, then $rc(G)\leq4n/5-1$.
\end{pro}

In this paper, we will study the rainbow connection number of a
connected graph with cut vertices and obtain a tight upper bound.
Besides, a tight upper bound for a 2-edge-connected (bridgeless)
graph is also obtained.

\section{Main results}

We first show that every 2-connected graph $G$ with odd number of
vertices has a rainbow edge-coloring with a nice property.

\begin{lem}\label{lem3}
Let $G$ be a 2-connected graph of order $n \ (n\geq3)$ and $v_0$ be
any vertex of $G$. If $n$ is odd, then $G$ has a rainbow
$\lceil\frac n 2\rceil$-edge-coloring $c$ such that there exists a
color $x$ of the edge-coloring satisfying that every vertex of $G$
can be connected by a rainbow path $P$ to $v_0$ with $x\notin c(P)$.
\end{lem}

\begin{proof}
Since $G$ is 2-connected, $G$ has a nonincreasing ear decomposition
$G_0, G_1,$ $\cdots,$ $G_q(=G) \ (q\geq0)$ satisfying that (1) $G_0$
is a cycle with $v_0\in V(G_0)$; (2) $G_i=G_{i-1}\bigcup P_i$, where
$P_i \ (1\leq i\leq q)$ is an ear of $G_{i-1}$ in $G$; (3)
$\ell(P_1)\geq\ell(P_2)\geq\cdots\geq\ell(P_q)$. We consider the
following two cases.

\noindent {\bf Case 1.} No ear of $P_1,\cdots, P_q$ has an even
length.

In this case, since $G$ has an odd order, $G_0$ must be an odd
cycle. Assume that $G_0=v_0v_1\cdots v_{2k}v_{2k+1}(=v_0)$ with
$k\geq1$. Define a $(k+1)$-edge-coloring $c_0$ of $G_0$ by
$c_0(v_{i-1}v_i)=x_i$ for $i$ with $1\leq i\leq k+1$ and
$c_0(v_{i-1}v_i)=x_{i-k-1}$ for $i$ with $k+2\leq i\leq 2k+1$. It
can be checked that $c_0$ is a rainbow $\lceil\frac {|V(G_0)|}
2\rceil$-edge-coloring of $G_0$ such that every vertex of $G_0$ can
be connected by a rainbow path $P$ in $G_0$ to $v_0$ with
$x_{k+1}\notin c_0(P)$. If $G_0=G$, the conclusion holds.

Now assume that $G_0\neq G$ and $P_1=v_0^\prime v_1^\prime\cdots
v_{2s}^\prime v_{2s+1}^\prime(s\geq0)$ with $V(G_0)\bigcap
V(P_1)=\{v_0^\prime, v_{2s+1}^\prime\}$. Define an edge-coloring
$c_1$ of $G_1=G_0\bigcup P_1$ by $c_1(e)=c_0(e)$ for $e\in E(G_0)$,
$c_1(v_{i-1}^\prime v_i^\prime)=y_i$ for $i$ with $1\leq i\leq s$,
$c_1(v_s^\prime v_{s+1}^\prime) =x^\prime$ and $c_1(v_{i-1}^\prime
v_i^\prime)=y_{i-s-1}$ for $i$ with $s+2\leq i\leq 2s+1$, where
$y_1, \cdots, y_s$ are new colors and $x^\prime$ is a color that
already appeared in $G_0$. Here, if $\ell(P_1)=1$, i.e., $s=0$, we
just color the only edge $v_0^\prime v_1^\prime$ of $P_1$ by a color
that appeared in $G_0$. It can be checked that $c_1$ is a rainbow
$\lceil\frac {|V(G_1)|} 2\rceil$-edge-coloring of $G_1$. From the
definition of $c_1$, every vertex of $G_0$ can be connected by a
rainbow path $P$ in $G_0$ to $v_0$ with $x_{k+1}\notin c_1(P)$. Let
$P^\prime$ and $P^{\prime\prime}$ be the rainbow paths,
respectively, from $v_0^\prime$ and $v_{2s+1}^\prime$ to $v_0$ in
$G_0$ such that $x_{k+1}\notin c_1(P^\prime)$ and $x_{k+1}\notin
c_1(P^{\prime\prime})$. For any vertex $v_j^\prime \ (1\leq j\leq
s)$, $v_j^\prime P_1v_0^\prime P^\prime v_0$ is a rainbow path in
$G_1$ from $v_j^\prime$ to $v_0$ such that $x_{k+1}\notin
c_1(v_j^\prime P_1v_0^\prime P^\prime v_0)$. For any vertex
$v_j^\prime \ (s+1\leq j\leq 2s)$, we can choose $v_j^\prime
P_1v_{2s+1}^\prime P^{\prime\prime}v_0$ as a rainbow path in $G_1$
from $v_j^\prime$ to $v_0$ such that $x_{k+1}\notin c_1(v_j^\prime
P_1v_{2s+1}^\prime P^{\prime\prime}v_0)$. Hence, $c_1$ is a required
rainbow edge-coloring of $G_1$.

If $G_1=G$, the conclusion holds. Otherwise, repeating the above
process of defining $c_1$ from $c_0$, we can obtain a rainbow
$\lceil\frac {|V(G_i)|} 2\rceil$-edge-coloring of $G_i \ (2\leq
i\leq q)$ such that every vertex of $G_i$ can be connected by a
rainbow path $P$ in $G_i$ to $v_0$ with $x_{k+1} \notin c_i(P)$.
Therefore, $c_q$ is a required rainbow $\lceil\frac n
2\rceil$-edge-coloring of $G$.

\noindent {\bf Case 2.} At least one of $P_1,\cdots, P_q$ has an
even length.

Suppose that $P_t \ (1\leq t\leq q)$ is the last added ear with an
even length. So $P_{t+1}, \cdots, P_s$ have odd lengths. From Case
1, we just need to show that $G_t$ has a required rainbow
$\lceil\frac{|V(G_t)|} 2\rceil$-edge-coloring. Now we will consider
the following two cases:

\noindent {\bf Subcase 2.1.} At most one of the ears $P_1, \cdots,
P_{t-1}$ has length 2.

Assume that $P_t=v_0^\prime v_1^\prime\cdots v_{2s-1}^\prime
v_{2s}^\prime$ such that $V(P_t)\bigcap V(G_{t-1})=\{v_0^\prime,
v_{2s}^\prime\}$. It is obvious that $G_0, G_1, \cdots, G_{t-1}$ is
a nonincreasing ear decomposition of $G_{t-1}$ with at most one ear
with length 2. From Lemmas \ref{lem1} and \ref{lem2}, $G_{t-1}$ has
an incomplete rainbow $\lceil\frac {|V(G_{t-1})|}
2\rceil$-edge-coloring $c_{t-1}$. In $G_{t-1}$, there exists an
incomplete rainbow path $P^\prime$ from $v_0$ to one of $v_0^\prime$
and $v_{2s}^\prime$ (say $v_{2s}^\prime$). Assume that $x^\prime$ is
a color of the coloring $c_{t-1}$ with $x^\prime\notin
c_{t-1}(P^\prime)$. Define an edge-coloring $c_t$ of
$G_t=G_{t-1}\bigcup P_t$ by $c_t(e)=c_{t-1}(e)$ for $e\in
E(G_{t-1})$, $c_t(v_{i-1}^\prime v_i^\prime)=x_i$ for $i$ with
$1\leq i\leq s$, $c_t(v_s^\prime v_{s+1}^\prime)=x^\prime$ and
$c_t(v_{i-1}^\prime v_i^\prime)=x_{i-s-1}$ for $i$ with $s+2\leq
i\leq 2s$, where $x_1, \cdots, x_s$ are new colors. It can be
checked that $c_t$ is a rainbow $\lceil\frac {|V(G_t)|}
2\rceil$-edge-coloring of $G_t$. From the definition of coloring
$c_t$, every vertex of $G_{t-1}$ has a rainbow path $P$ in $G_{t-1}$
to $v_0$ with $x_s\notin c_t(P)$. Let $P^{\prime\prime}$ be a
rainbow path in $G_{t-1}$ from $v_0^\prime$ to $v_0$. For any vertex
$v_j^\prime \ (1\leq j\leq s-1)$, $v_j^\prime P_tv_0^\prime
P^{\prime\prime}v_0$ is a rainbow path in $G_t$ from $v_j^\prime$ to
$v_0$ such that $x_s\notin c_t(v_j^\prime P_tv_0^\prime
P^{\prime\prime}v_0)$. For any vertex $v_j^\prime \ (s\leq j\leq
2s-1)$, we have $v_j^\prime P_tv_{2s}^\prime P^\prime v_0$ is a
rainbow path in $G_t$ from $v_j^\prime$ to $v_0$ such that
$x_s\notin c_t(v_j^\prime P_tv_{2s}^\prime P^\prime v_0)$. So every
vertex of $G_t$ has a rainbow path $P$ in $G_t$ to $v_0$ with
$x_s\notin c_t(P)$. Hence, $c_t$ is a required rainbow edge-coloring
of $G_t$.

\noindent {\bf Subcase 2.2.} At least two ears of $P_1, \cdots,
P_{t-1}$ have length 2.

In this case, it is obvious that $\ell(P_t)=2$ and
$\ell(P_{t+1})=\cdots =\ell(P_q)=1$. Assume that
$\ell(P_1)\geq\cdots \geq\ell(P_h)\geq3$ and $\ell(P_{h+1})=\cdots
=\ell(P_t)=2$. Here at least three ears have length 2, i.e.,
$t-h\geq3$. From Theorem \ref{thm:2con}, $G_h$ has a rainbow
$\lceil\frac {|V(G_h)|} 2\rceil$-edge-coloring $c_h$. Assume that
$P_j=a_jv_jb_j \ (h+1\leq j\leq t)$ such that $V(P_j)\bigcap
V(G_h)=\{a_j, b_j\}$. Define an edge-coloring $c_t$ of $G_t$ by
$c_t(e)=c_h(e)$ for $e\in E(G_h)$, $c_t(a_jv_j)=x_1$ for $j$ with
$h+1\leq j\leq t$ and $c_t(v_jb_j)=x_2$ for $j$ with $h+1\leq j\leq
t$, where $x_1, x_2$ are new colors. It is easy to check that $c_t$
is a rainbow edge-coloring of $G_t$ with at most $\lceil\frac
{|V(G_t)|} 2\rceil$ colors and every vertex of $G_t$ has a rainbow
path $P$ to $v_0$ with $x_2\notin c_t(P)$. Therefore, $G_t$ has a
required rainbow $\lceil\frac{|V(G_t)|} 2\rceil$-edge-coloring.
\end{proof}

\begin{thm}\label{them1}
Let $G$ be a connected graph of order $n \ (n\geq3)$ and $G$ has a
block decomposition $B_1,\cdots, B_q \ (q\geq2)$, where $r$ blocks
are even blocks and the others are odd ones. Then
$rc(G)\leq\frac{n+r-1}{2}$ and the upper bound is tight.
\end{thm}

\begin{proof}
Let $G$ be a connected graph of order $n$ with $q \ (q\geq2)$ blocks
in its block decomposition. If $G$ has at least one even block, we
choose $G_1=B_1$ being an even block of $G$; otherwise, $G_1=B_1$
being an odd block of $G$. Since $q\geq2$ and $G$ is connected, $G$
has a block $B_2$ such that $V(G_1)\bigcap V(B_2)=\{v_1\}$. Let
$G_2=G_1\bigcup B_2$. So $G_2$ is a connected graph which consists
of two blocks $B_1, B_2$. Repeating the process of adding $B_2$ to
$G_1$, we obtain a sequence of subgraphs $G_1, G_2,\cdots, G_q$ such
that $G_i \ (1\leq i\leq q)$ is a connected graph and
$G_i=B_1\bigcup B_2\bigcup \cdots\bigcup B_i \ (2\leq i\leq q)$ with
$V(G_{i-1})\bigcap V(B_i)=\{v_{i-1}\}$ for $i$ with $2\leq i\leq q$.
Denote the order of $B_i \ (1\leq i\leq q)$ by $n_i$. From Theorem
\ref{thm:2con} and $rc(K_2)=1$, every block $B$ has a rainbow
$\lceil\frac {|V(B)|} 2\rceil$-edge-coloring. We will consider the
following two cases.

\noindent {\bf Case 1.} $r\geq1$.

From the definition of $G_1$, $G_1=B_1$ is an even block and $G_1$
has a rainbow $\lfloor\frac{n_1} 2\rfloor$-edge-coloring $c_1$. If
$B_2$ is an even block, color the edges of $B_2$ with
$\lfloor\frac{n_2} 2\rfloor$ new colors such that $B_2$ is rainbow
connected. It is obvious that $G_2$ is rainbow connected and the
obtained edge-coloring $c_2$ of $G_2$ uses $\lfloor\frac{n_1}
2\rfloor+ \lfloor\frac{n_2} 2\rfloor$ colors. Consider the case that
$B_2$ is an odd block. From Lemma \ref{lem3}, $B_2$ has a rainbow
edge-coloring $c_2^\prime$ with $\lceil\frac{n_2} 2\rceil$ new
colors such that there exists a color $x^\prime$ of $c_2^\prime$
satisfying that every vertex of $B_2$ has a rainbow path $P$ in
$B_2$ to $v_1$ with $x^\prime\notin c_2^\prime(P)$. Replacing the
color $x^\prime$ of $c_2^\prime$ by a color $x$ that already
appeared in $G_1$, we obtain an edge-coloring $c_2$ of $G_2$ with
$\lfloor\frac{n_1} 2\rfloor+\lfloor\frac{n_2} 2\rfloor$ colors. It
is obvious that $G_1$ and $B_2$ are rainbow connected, respectively.
Consider two vertices $v^\prime\in V(G_1)$ and $v^{\prime\prime}\in
V(B_2)$. From the definition of $c_2$, there are two rainbow paths
$P^\prime$ in $G_1$ from $v^\prime$ to $v_1$ and $P^{\prime\prime}$
in $B_2$ from $v_1$ to $v^{\prime\prime}$ such that $x\notin
c_2(P^{\prime\prime})$. So $v^\prime P^\prime
v_1P^{\prime\prime}v^{\prime\prime}$ is a rainbow path from
$v^\prime$ to $v^{\prime\prime}$ in $G_2$. Hence, $c_2$ is a rainbow
edge-coloring of $G_2$ with $\lfloor\frac{n_1}
2\rfloor+\lfloor\frac{n_2} 2\rfloor$ colors.

If $q\geq3$, we can repeat the process of defining $c_2$ from $c_1$ to
obtain a rainbow edge-coloring $c_q$ of $G_q(=G)$ with $\lfloor\frac{n_1}
2\rfloor+\lfloor\frac{n_2} 2\rfloor+\cdots+\lfloor\frac {n_q} 2\rfloor$ colors.

\noindent {\bf Case 2.} $r=0$.

In this case, $G_2=B_1\bigcup B_2$ consists of two odd blocks. From
Lemma \ref{lem3}, $B_i \ (i=1, 2)$ has a rainbow $\lceil\frac {n_i}
2\rceil$-edge-coloring $c_i^\prime$ such that $x_i^\prime$ is a
color of $c_i^\prime$ satisfying that every vertex of $B_i \ (i=1,
2)$ has a rainbow path $P$ in $B_i$ to $v_1$ with $x_i^\prime\notin
c_i^\prime(P)$. Note that $c_1^\prime (B_1)\bigcap
c_2^\prime(B_2)=\emptyset$. Assume that $x_i \ (i=1, 2)$ is a color
of $c_i^\prime$ such that $x_i\neq x_i^\prime$. Replacing
$x_1^\prime$ by $x_2$ in $B_1$ and $x_2^\prime$ by $x_1$ in $B_2$,
we obtain an edge-coloring $c_2$ of $G_2$ with $\lfloor\frac{n_1}
2\rfloor+\lfloor\frac{n_2} 2\rfloor$ colors. It is obvious that $B_i
\ (i=1, 2)$ is rainbow connected. Consider two vertices $v^\prime\in
V(B_1)$ and $v^{\prime\prime}\in V(B_2)$. From the definition of
$c_2$, there exist two rainbow paths $P^\prime$ in $B_1$ from
$v^\prime$ to $v_1$ and $P^{\prime\prime}$ in $B_2$ from $v_1$ to
$v^{\prime\prime}$ such that $x_2\notin c_2(P^\prime)$ and
$x_1\notin c_2(P^{\prime\prime})$. So $v^\prime P^\prime
v_1P^{\prime\prime}v^{\prime\prime}$ is a rainbow path in $G_2$ from
$v^\prime$ to $v^{\prime\prime}$. Hence, $c_2$ is a rainbow
edge-coloring of $G_2$ with $\lfloor\frac{n_1} 2\rfloor+
\lfloor\frac{n_2} 2\rfloor$ colors. If $q\geq3$, we can color the
blocks $B_3,\cdots, B_q$ similar to Case 1 to obtain a rainbow
edge-coloring of $G$ with $\lfloor\frac{n_1}
2\rfloor+\lfloor\frac{n_2} 2\rfloor+\cdots+\lfloor\frac {n_q}
2\rfloor$ colors.

Therefore, in any case we have that $rc(G)\leq\lfloor\frac{n_1}
2\rfloor+\lfloor\frac{n_2} 2\rfloor+\cdots+\lfloor\frac {n_q}
2\rfloor=\frac{n+r-1} 2$.

\begin{figure}[h,t,b,p]
\begin{center}
\scalebox{0.45}[0.45]{\includegraphics{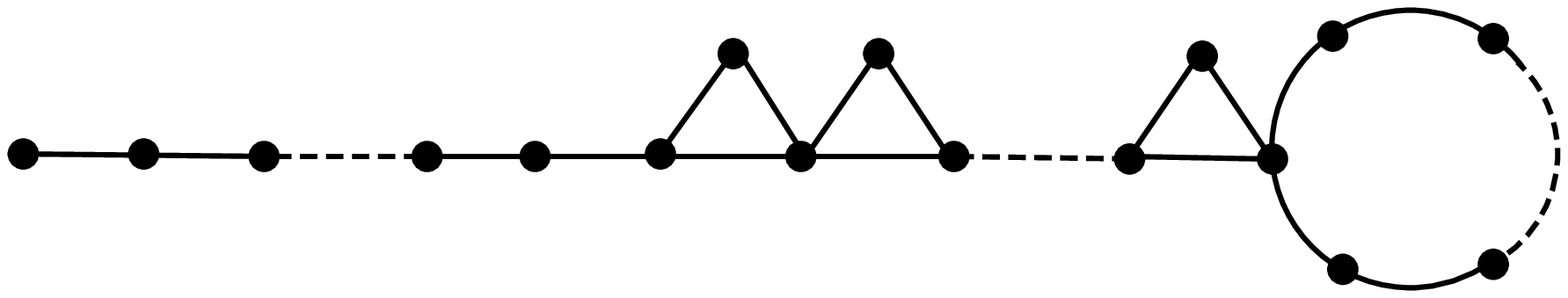}}

Figure 1. Graphs for the tightness of Theorem \ref{them1}.
\end{center}
\end{figure}

In order to prove that the upper bound is tight, we will show that
for any integers $n, r, q$, if there exist graphs of order $n$ with
$r$ even blocks and $q-r$ odd blocks, then one of such graphs has a
rainbow connection number $\frac{n+r-1} 2$.

In fact, if there exists a connected graph of order $n$ with $r$ even blocks, then $n+r$
must be an odd number. The graph $G$ of order $n$ in Figure 1
consists of $r$ even blocks $K_2$, $q-r-1$ odd cycles $K_3$ and
one odd cycle $C_{n-2q+r+2}$. Since $d(G)=\frac{n+r-1} 2$ and $d(G)\leq
rc(G)\leq \frac{n+r-1} 2$, we have $rc(G)=\frac{n+r-1} 2$.
\end{proof}

In the following, we give a tight upper bound of the rainbow
connection number for a 2-edge-connected graph which improves the
result of Proposition \ref{pro1}.

\begin{thm}\label{thm2}
Let $G$ be a 2-edge-connected graph of order $n \ (n\geq3)$. Then
$$rc(G)\leq\left\{
\begin{array}{ll}
2k   & $if $ n=3k+1 $or $ 3k+2 \\
2k+1 & $if $ n=3k+3
\end{array},
\right.$$

and the upper bound is tight.
\end{thm}

\begin{proof}
Suppose that $G$ has the block decomposition $B_1, B_2, \cdots,
B_q$. Since $G$ is 2-edge-connected, we have $|B_i|\geq3, 1\leq
i\leq q$. And if $B_i$ is an even block, then $|B_i|\geq4$. If $G$
has $r$ even blocks, then $3r+1\leq n$, i.e., $r\leq\frac{n-1} 3$.
From Theorem \ref{them1}, $rc(G)\leq\frac{n+r-1} 2\leq\frac{2n-2} 3$.
Since $rc(G)$ is an integer,
$rc(G)\leq\left\{
\begin{array}{ll}
2k   & $if $ n=3k+1$ or $ 3k+2 \\
2k+1 & $if $ n=3k+3
\end{array}.
\right.$

\begin{figure}[h,t,b,p]
\begin{center}
\scalebox{0.45}[0.45]{\includegraphics{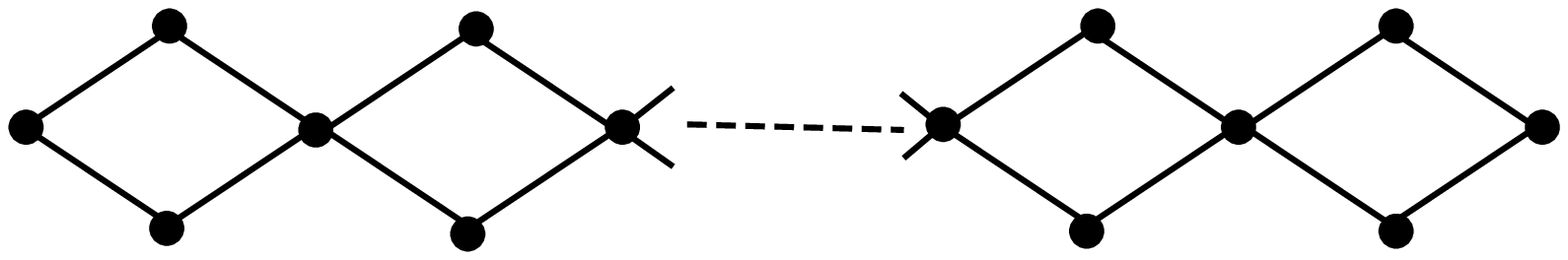}}

$G_1$

\scalebox{0.45}[0.45]{\includegraphics{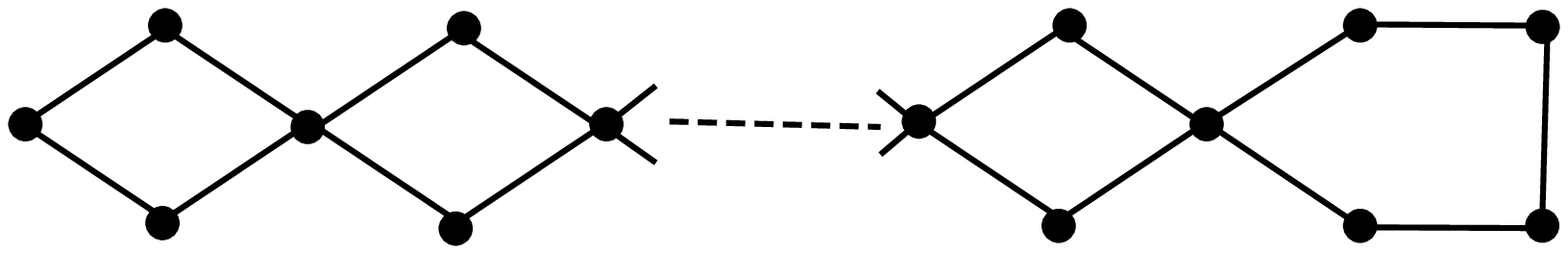}}

$G_2$

\scalebox{0.45}[0.45]{\includegraphics{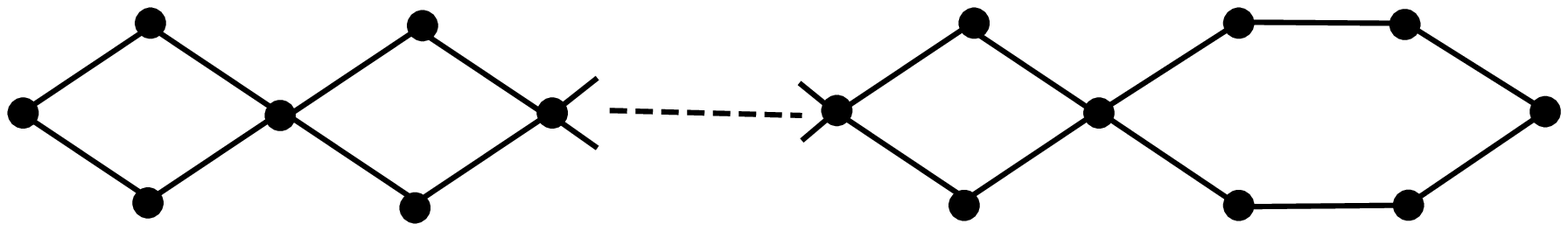}}

$G_3$

Figure 2. Graphs for the tightness of Theorem \ref{thm2}.
\end{center}
\end{figure}

The three graphs $G_1, G_2, G_3$ in Figure 2 are 2-edge-connected.
The order of $G_i \ (i=1, 2, 3)$ is $3k+i$, and $d(G_1)=d(G_2)=2k$
and $d(G_3)=2k+1$. From the above result and $d(G)\leq rc(G)$, we
have that $rc(G_1)=rc(G_2)=2k$ and $rc(G_3)=2k+1$. Hence, the upper
bound is tight.
\end{proof}

\end{document}